\documentclass[11pt]{amsart}

\usepackage{amscd,amsmath,amssymb,fancyhdr,color}
\usepackage[utf8x]{inputenc}
\usepackage{amsfonts}
\usepackage{amsthm}
\usepackage{accents}
\usepackage{graphicx}
\usepackage{float}
\usepackage{subcaption}
\usepackage{verbatim}
\usepackage{indentfirst}
\usepackage{dsfont}
\usepackage{tikz}
\usetikzlibrary{matrix}
\usepackage{enumerate}
\usepackage{extarrows}

\usepackage{faktor}

\usepackage[backref=page]{hyperref}
\renewcommand*{\backref}[1]{}
\renewcommand*{\backrefalt}[4]{%
    \ifcase #1 (Not cited.)%
    \or        (Cited on page~#2.)%
    \else      (Cited on pages~#2.)%
    \fi}

\hypersetup{
     colorlinks   = true,
     citecolor    = magenta
}

\DeclareMathOperator{\Ann}{Ann}
\DeclareMathOperator{\Ker}{Ker}

\numberwithin{equation}{section}

\def\eqref#1{(\ref{#1})}

\newcommand{\Z}{{\mathbb Z}}

\newcommand{\R}{{\mathbb R}}

\def\1{\sqrt{-1}\:}
\newcommand{\restrict}[1]{{\big|_{{\phantom{|}\!\!}_{#1}}}}
\newcommand{\cntrct}                
{\hspace{2pt}\raisebox{1pt}{\text{$\lrcorner$}}\hspace{2pt}}

\newcommand{\codim}{\operatorname{codim}}
\renewcommand{\dim}{\operatorname{dim}}

\renewcommand{\Im}{\operatorname{Im}}

\newcommand{\ie}{{\em i.e. }}
\newcommand{\eg}{{\em e.g. }}

\renewcommand{\to}{\longrightarrow}

\newcounter{Mycounter}[section]
\newcounter{lemma}[section]
\newcounter{claim}[section]

\newcounter{sublemma}[section]

\newcounter{corollary}[section]

\newcounter{theorem}[section]

\newcounter{conjecture}[section]

\newcounter{proposition}[section]

\newcounter{definition}[section]

\newcounter{example}[section]

\newcounter{remark}[section]

\newcounter{problem}[section]

\newcounter{question}[section]

\newcounter{theoremnonmb}[section]

\makeatletter

\@addtoreset{equation}{section}

\@addtoreset{footnote}{section}

\makeatother

\usetikzlibrary{arrows,chains,matrix,positioning,scopes}

\makeatletter
\tikzset{join/.code=\tikzset{after node path={%
			\ifx\tikzchainprevious\pgfutil@empty\else(\tikzchainprevious)%
			edge[every join]#1(\tikzchaincurrent)\fi}}}
\makeatother

\tikzset{>=stealth',every on chain/.append style={join},
	every join/.style={->}}

\begin{document}
	
\newpage


\title[LCS reduction of the cotangent bundle]{Locally conformally symplectic reduction of the cotangent bundle}
\author{Miron Stanciu}
\address{Institute of Mathematics ``Simion Stoilow'' of the Romanian Academy\newline 
	21 Calea Grivitei Street, 010702, Bucharest, Romania\newline 
	\indent	{\em and}\newline
	\indent	University of Bucharest, Faculty of Mathematics and Computer Science, 14 Academiei Str., Bucharest, Romania}
\thanks{\textbf{Declarations:} \\[.05in]
	{\bf Funding:} Partially supported by a grant of Ministry of Research and Innovation, CNCS - UEFISCDI,
	project number PN-III-P4-ID-PCE-2016-0065, within PNCDI
	III.\\
	{\bf Conflicts of interest:} Not applicable.\\
	{\bf Availability of data, material and code:} Not applicable.\\[.1in]
}
\email{miron.stanciu@imar.ro; mirostnc@gmail.com}

\date{\today}

\begin{abstract}
In \cite{stnc}, we introduced a reduction procedure for locally conformally symplectic manifolds at any regular value of the natural momentum mapping. We use this construction to prove an analogue of a well-known theorem in the symplectic setting about the reduction of cotangent bundles.
\end{abstract}

\maketitle

\hypersetup{linkcolor=blue}
\tableofcontents

\section{Introduction}
Equivariant symplectic reduction is a classical procedure (see \cite[pp. 298-299]{am}, \cite[ch. 7]{br}, \cite{McD-S}) through which one obtains symplectic manifolds by factoring the level sets of the natural momentum mapping given by the Hamiltonian action of a Lie group on an original symplectic manifold. Its importance in understanding symplectic geometry and topology, as well in obtaining new examples, is hard to overestimate.

Before a general reduction theory was properly developed, the reduction of the cotangent bundle was studied by Smale, \cite{smale}, and it was Marsden and Weinstein, \cite{mw1}, \cite{mw2}, who applied further techniques to different physical systems which admitted noncanonical Poisson bracket formalisms, such as the rigid body, fluid or magnetohydrodynamic systems. For an in depth historical overview of symplectic reduction and its deep ties to physical phenomena, see \cite[Chapter 1]{ratiu}.

The canonical symplectic structure of the cotangent bundle of any differentiable manifold has a kind of universality property with respect to reduction. Namely, if a group $G$ acts on a manifold $Q$ such that the quotient $Q/G$ is a manifold, the action can be naturally lifted to a Hamiltonian action on the cotangent bundle $T^*Q$ and the reduction at 0 of $T^*Q$ is symplectomorphic with the cotangent bundle $T^*(Q/G)$. When performing reduction at a non-zero regular value of the natural momentum map, the symplectomorphism becomes a symplectic embedding, see 
\ref{thm:embeddingRatiu} where we reproduce the theorem proved in \cite[Theorem 2.2.1]{ratiu}.

Locally conformally symplectic (LCS for short) manifolds are quotients of symplectic manifolds by a discrete group of homotheties of the symplectic form. Alternatively, they are manifolds which are locally conformal to symplectic manifolds; this can be reformulated to require that a non-degenerate $2$-form satisfy a differential equation. Although the definition (not explicitly stated) appears in 1943, \cite{lee}, this geometry is widely studied only since the seminal work of I. Vaisman, \cite{va_lcs}, in which the fact that they are natural phase spaces of Hamiltonian dynamical systems, more general than the symplectic manifolds, is also pointed out, thereby giving a physics motivation behind their further study. For a survey on recent progress in LCS geometry, see \cite{baz}. A result due to  Eliashberg \& Murphy, \cite{el}, showed that LCS structures exist on all compact almost complex manifolds with non-trivial one-cohomology (see Section \ref{sectionPrelim} for a precise statement). This gave new impetus on the research on LCS manifolds.

Symplectic reduction  was first adapted to the LCS setting by Haller \& Rybicki, \cite{hr}. They introduced a method of reducing submanifolds of LCS manifolds with special properties, but which does not apply, generally, to level sets of natural momentum maps determined by a Lie group action. In \cite{stnc}, we introduced an LCS reduction scheme which generalizes the equivariant symplectic method and works, with similar hypotheses, for any regular value of the natural momentum map.

On the other hand, the cotangent bundle of a manifold has  many LCS structures, given by  choices of a closed $1$-form on the manifold (this was first noticed by \cite[Example 3.1]{hr2}). It is then natural to try to adapt the symplectic cotangent reduction theorems to the LCS setting. A useful  observation is that if a Lie group acts on a  manifold, determining a principal bundle, then the corresponding action on the cotangent bundle has the same natural momentum map considered with the symplectic structure, as well as with any LCS structure.

Our main result in this paper (\ref{thm:embedding}) is that the universality result of \cite[Theorem 2.2.1]{ratiu}, reproduced in \ref{thm:embeddingRatiu}, is true, in the same conditions, for the LCS structures of a cotangent bundle, reduced as in \cite{stnc}, at any regular value. As the precise statement depends on a few definitions and notations which will be presented in the preliminaries, for now we can only give an outline of the result at regular value zero:
\begin{theoremnonmb}
	Let $Q$ be a manifold with a free and proper action of a Lie group $G$ and take $\overline{\theta} \in \Omega^1(Q/G)$ a closed $1$-form. Denote by $\theta = p^* \overline{\theta}$, where $p:Q \to Q/G$ is the natural projection.
	
	By the result of \cite{hr2}, these $1$-forms determine two explicit LCS structures: one on $T^*Q$ given by $\theta$, and one on $T^*(Q/G)$ given by $\overline{\theta}$.
	
	Moreover, $G$ has an action on $T^*Q$ which is compatible with its LCS structure. Let $\mu:T^*Q \to \mathfrak{g}^*$ be the corresponding natural momentum mapping.
	
	Then, denoting by $(T^*Q)_0$ the LCS reduction at regular value $0$ of $T^*Q$, there is a canonical and explicit LCS isomorphism 
	\[
	(T^*Q)_0 \simeq T^*(Q/G).
	\]
	
\end{theoremnonmb}

For a general value $\xi$, additional hypotheses have to be assumed to obtain the desired result, namely that there is an LCS embedding between the natural LCS manifolds that occur. We follow the main result with a discussion of those hypotheses, providing both examples and counterexamples of cases in which they cannot hold and moreover the result itself is not true.

\bigskip

In addition to the result itself, this shows the naturality of the reduction scheme introduced in \cite{stnc}.

The paper is organized as follows. In Section \ref{sectionPrelim} we introduce the necessary background for LCS geometry and recall our constructions in \cite{stnc}. Section \ref{main_section} recalls the statement of the cotangent reduction theorems in symplectic geometry, then presents the statement and proof of our main result. The paper ends with the description of several examples which show the necessity of the hypotheses in our statement.

\section{Twisted Hamiltonian actions and LCS reduction}
\label{sectionPrelim}

\subsection{Background on LCS geometry}
We use this subsection to explain the setting for our theorem, starting with a few basic definitions .

\begin{definition}
	A manifold $M$ with a non-degenerate two-form $\omega$ is called \textit{locally conformally symplectic} (for short, \textit{LCS}) if there exists a closed one-form $\theta$ such that 
	$$d\omega = \theta \wedge \omega.$$
	In that case, $\omega$ is called the LCS form and $\theta$ is called {\em the Lee form} of $\omega$. 
\end{definition}

\bigskip

It was recently proven by Y. Eliashberg and E. Murphy that LCS manifolds are widespread:

\begin{theorem}\textnormal {(\cite{el})}
	\label{thm:eliashberg}
	Let $(M, J)$ be a closed $2n$-dimensional almost complex manifold and $[\eta]$ a non-zero cohomology class in $H^1(M, \Z)$. Then there exists a locally conformally symplectic structure $(\omega,\theta)$ in a formal homotopy class determined by $J$, with $\theta = c \cdot \eta$ for some real $c \neq 0$. 
\end{theorem}

\begin{definition}
	Let $(M, \omega, \theta)$ be an LCS manifold. Define the \textit{twisted de Rham operator}
	\[
	d_\theta = d - \theta \wedge \ : \Omega^\bullet(M) \to \Omega^{\bullet + 1}(M).
	\]
	In particular, $d_\theta\omega=0$.	Since $d\theta = 0$, we have $d_\theta^2 = 0$, so this operator defines the so-called {\em twisted} or {\em Lichnerowicz} cohomology
	\[
	H^\bullet_\theta (M)= \frac{\Ker d_\theta}{\Im d_\theta}.
	\]
\end{definition}

\begin{definition}
	On an LCS manifold $(M, \omega, \theta)$, we can introduce a new differential operator:
	\[
	\mathcal{L}^\theta_X : \Omega^\bullet(M) \to \Omega^\bullet(M)
	\]
	the twisted Lie derivative along the vector field $X$, via the formula
	\[
	\mathcal{L}_X^\theta \alpha = \mathcal{L}_X \alpha - \theta(X) \alpha.
	\]
	
	Using the usual Cartan equality, we see that $\mathcal{L}^\theta_X$ satisfies a similar equality for the operator $d_\theta$, namely
	\[
	\mathcal{L}_X^\theta \alpha = d_\theta i_{X} \alpha + i_{X} d_\theta \alpha.
	\]
\end{definition}

\begin{definition}
	In a symplectic vector space $(V, \omega)$, given a subspace $W \le V$, we denote by
	\[
	W^\omega = \{ v \in V \ | \ \omega (v, W) = 0 \} \le V
	\]
	the $\omega$-dual of $W$. Since $\omega$ is non-degenerate, we have 
	\[
	\dim W + \dim W^\omega = \dim V.
	\]
	Note that the above sum is not necessarily direct.
\end{definition}

\begin{remark}
	The LCS condition is conformally invariant in the following sense. If $\omega$ is an LCS form on $M$ with Lee form $\theta$, then $e^f \omega$ is also an LCS form with Lee form $\theta + df$, for any $f \in \mathcal{C}^\infty(M)$. Therefore, it is sometimes more interesting to work with an LCS structure 
	\[
	[(\omega, \theta)] = \{ (e^f \omega, \theta + df) \ | \ f \in \mathcal{C}^\infty(M) \}
	\]
	on $M$ rather than with particular forms in that structure. Note that there exists a symplectic form conformal to a given LCS form with Lee form $\theta$ if and only if $[\theta] = 0 \in H^1(M)$; in this case, the structure is called {\em globally conformally symplectic (GCS)}. 
\end{remark}

\bigskip

The following example gives a class of LCS structures on any cotangent bundle and will provide the setting for the main theorem of this article:

\begin{example}(\cite{hr2})
	\label{ex:cotangentLCS}
	In addition to the canonic symplectic structure, the cotangent bundle of a given manifold $Q$ can be endowed with LCS structures: Consider the tautological $1$-form $\eta$ on $T^*Q$, defined in any point $\alpha_x \in T^*_xQ$ by
	\[
	\eta_{\alpha_x} (v) = \alpha_x (\pi_* v).
	\]
	Then, for any $\theta$ a closed $1$-form on $Q$, $\omega_\theta = d_{\pi^* \theta} \eta$ is an LCS form on $T^*Q$ (which is not globally conformally symplectic unless $\theta$ is exact).
\end{example}

\bigskip

\subsection{Twisted Hamiltonian actions}

We now define the notions of twisted Hamiltonian actions and of natural momentum mappings (these are the same used by \cite{va_lcs}, \cite{hr}, \cite{bgp} and \cite{nico}, among others).

\begin{definition}
	Let $(M, \omega, \theta)$ be an LCS manifolds and $G$ a Lie group acting on it. 
	
	\begin{enumerate}[i)]
		\item For any $a \in \mathfrak{g}$, denote by $X_a$ the fundamental vector field on $M$ associated to $a$,
		\[
		(X_a)_x = \frac{d}{dt}_{|t=0} (e^{ta} \cdot x).
		\]
		\item The action is called \textit{twisted symplectic} if 
		\[
		g^* \omega = e^{\varphi_g} \omega \text{ and } g^* \theta = \theta + d\varphi_g, \forall g \in G,
		\]
		for $\varphi_g$ some smooth function on $M$.
		
		\item The action is called \textit{twisted Hamiltonian} if, in addition, for any $a \in \mathfrak{g}$, the $1$-form $i_{X_a} \omega$ is $d_\theta$-exact, say
		\[
		d_\theta \rho_a = i_{X_a} \omega.
		\] 
	\end{enumerate}
\end{definition}

\begin{remark}
	\label{rem:natural}
	If $\theta$ is not exact, then the functions $\rho_a$ as above are uniquely determined (see \eg \cite[Proposition $2.4$]{stnc}). As the non-symplectic setting is the interesting one for us, we will assume this to be the case.
\end{remark}

\begin{proposition}
\label{prop:theta(Xa)}
	If $\omega$ (and thus $\theta$) is $G$-invariant and the action is twisted Hamiltonian, then $\theta(X_a) = 0$ for all $a \in \mathfrak{g}$.
	
	\begin{proof}
		Indeed,
		\[
		0 = d_\theta i_{X_a}\omega = \mathcal{L}_{X_a}^\theta \omega - i_{X_a} d_\theta \omega = - \theta(X_a) \omega.
		\]
		Since $\omega$ is non-degenerate, this implies $\theta(X_a) = 0$, as required.
	\end{proof}
\end{proposition}

\begin{definition}
	For a twisted Hamiltonian action of $G$ on $(M, \omega, \theta)$, keeping the notations above, define the \textit{natural momentum mapping} to be:
	\[
	\mu:M \to \mathfrak{g}^*, \ \mu(x)(a) = \rho_a(x), \ \forall a \in \mathfrak{g}.
	\]
\end{definition}

\begin{remark}
	Since we operate under the assumption that $\theta$ is not exact, the above momentum mapping is indeed natural, as it does not depend on any choice for the functions $\rho_\theta$ (see \ref{rem:natural} as opposed to the symplectic case, where the momentum mapping is not generally unique, see \eg \cite[Lecture 7]{br}). From now on, we will refer to it in short as the momentum mappping.
\end{remark}

\begin{proposition}
\label{rem:twistatHamcuOmegaExact}
	Let $(M, \omega, \theta)$ be an LCS manifold with a group action $G$. Assume $\omega = d_\theta \eta$, with $\eta$ and $\theta$ $G$-invariant. 
	
	Then the action is twisted Hamiltonian with momentum mapping
	\[
	\mu(x)(a) = -\eta_x ((X_a)_x).
	\]
	
	\begin{proof}
		We have, by \ref{prop:theta(Xa)},
		\[
		0 = \mathcal{L}_{X_a} \eta = \mathcal{L}^\theta_{X_a} \eta = d_\theta (\eta(X_a)) + i_{X_a}\omega \implies i_{X_a} \omega = - d_\theta(\eta(X_a)).
		\]
\end{proof}	
\end{proposition}

\subsection{Lifting of actions to the LCS cotangent bundle}

We now use \ref{rem:twistatHamcuOmegaExact} to make a quick computation in order to determine the momentum mapping for the action of a group on the cotangent bundle endowed with an LCS structure, as in \ref{ex:cotangentLCS}.

Recall that, given an action of a Lie group $G$ on a manifold $Q$, one can define a natural action of $G$ on the cotangent bundle $T^*Q$ via push-forward:
\[
g \cdot \alpha_q = (g^{-1})^* (\alpha_q) \in T^*_{g \cdot q}Q, \text{ for any } q \in Q.
\]

\begin{proposition}
\label{prop:momentumCTG}
	Consider $Q$ a manifold and $G$ a connected Lie group acting freely and properly on it. Let $\theta$ be a closed $1$-form on $Q$ such that $\theta(X_a) = 0$, for all $a \in \mathfrak{g}$.
	
	\begin{enumerate}
		\item[\rm{i)}] Denote by $\tilde{X}_a$ the fundamental vector fields of the natural action of $G$ on $T^*Q$ by push-forward. Then $\pi_* \tilde{X}_a = X_a$.
		\item[\rm{ii)}] This action is twisted Hamiltonian.
		\item[\rm{iii)}] The momentum mapping is given by the formula
		\[
		\mu:T^*Q \to \mathfrak{g}^*, \ \mu(\alpha)(a) = -\alpha(X_a).
		\]
		and all its values are regular.
	\end{enumerate}

	\begin{proof}
		We compute the projection of the fundamental vector fields:
		\begin{equation*}
		\begin{split}
		\pi_* ((\tilde{X}_a)_{\alpha_q}) &= \pi_* (\frac{d}{dt}{_\restrict{t=0}} e^{ta} \cdot \alpha_q) = \frac{d}{dt}_{t=0} (\pi(e^{ta} \cdot \alpha_q)) \\
		&= \frac{d}{dt}_{t=0} (e^{ta} \cdot q) = (X_a)_q.
		\end{split}
		\end{equation*}
		
		For the last two items, we use the criterion given in \ref{rem:twistatHamcuOmegaExact}. Indeed, $\omega_\theta = d_{\pi^* \theta} \eta$ and the tautological form $\eta$ is $G$-invariant. Furthermore,
		\[
		\mathcal{L}_{X_a} \theta = d(\theta(X_a)) + i_{X_a} d\theta = 0,
		\]
		so $\theta$ is also $G$-invariant. By \ref*{rem:twistatHamcuOmegaExact}, the action is twisted Hamiltonian with momentum mapping
		\[
		\mu(\alpha)(a) = -\eta_{\alpha} ((\tilde{X}_a)_\alpha) = -\alpha(X_a).
		\]
		Notice that $\mu$ does not depend on the chosen $\theta \in \Omega^1(Q)$.
		
		Take $\xi \in \mathfrak{g}^*$. For $\xi$ to be a regular value, we must show that $\codim \Ker d_\alpha \mu = \dim G$ for any $\alpha \in \mu^{-1}(\xi)$. Take $\omega_0 = d\eta$. Then a simple calculation shows
		\[
		(\Ker d \mu)^{\omega_0} = \{ X_a \ | \ a \in \mathfrak{g} \}.
		\]
		As the action is free and proper, the right hand side is of dimension $\dim G$.
	\end{proof}
	
\end{proposition}

\bigskip

\subsection{The LCS reduction method}

In order to make this article more self-contained, we end by reproducing the LCS reduction theorem proven in \cite{stnc}, accompanied by a few technical remarks we will need later.

\begin{theorem}\textnormal{(\cite[Theorem 3.15]{stnc})}
	\label{th:LCSred}
	Let $(M, \omega, \theta)$ be a connected LCS manifold and $G$ a connected Lie group acting twisted Hamiltonian on it.
	
	Let $\mu$ be the momentum mapping and $\xi \in \mathfrak{g}^*$ a regular value. Denote by $\mathcal{F} = T\mu^{-1}(\xi) \cap (T\mu^{-1}(\xi))^\omega$. Assume that one of the following conditions is met:
	\begin{itemize}
		\item The action of $G$ preserves the LCS form $\omega$.
		\item $\xi \wedge \theta_x(X_\cdot) = 0$ for all $x \in \mu^{-1}(\xi)$ and there exists a function $h$ on $\mu^{-1}(\xi)$ such that $\theta_{| \mathcal{F}} = dh$.
	\end{itemize}

	If $M_\xi := \faktor{\mu^{-1}(\xi)}{\mathcal{F}}$ is a smooth manifold and $\pi : \mu^{-1}(\xi) \to M_\xi$ is a submersion, then $M_\xi$ has an LCS structure such that the LCS form $\omega_\xi$ satisfies
	\[
	\pi^* \omega_\xi = e^f \omega_{|\mu^{-1}(\xi)}
	\]
	for some $f \in C^\infty(\mu^{-1}(\xi))$. 
	
	Moreover, one can take $f = h$; in particular, $f = 0$ if the action preserves the LCS form.
\end{theorem}

\begin{remark}
\label{rmk:explicatiiLCSred}
In the above setting, $\mathcal{F} = T\mu^{-1}(\xi) \cap (T\mu^{-1}(\xi))^\omega$ is, by definition, a distribution. It can, however, be shown to be involutive (see \cite[Remark 3.2]{stnc}), hence it generates a foliation, and whenever we quotient a set by $\mathcal{F}$ we mean its quotient by the corresponding foliation.

Moreover, the following can be proven (see \cite[Equation 3.6]{stnc}):
\[
(T\mu^{-1}(\xi))^\omega = \{ X_a + \xi(a) \theta^\omega \ | \ a \in \mathfrak{g} \}
\]
and 
\[
T\mu^{-1}(\xi) \cap (T\mu^{-1}(\xi))^\omega = \{ X_a + \xi(a) \theta^\omega \ | \ a \in \mathfrak{g}_\xi \},
\]
where $\omega^\theta$ is the $\omega$-dual of $\theta$ \ie $\omega(\theta^\omega, \cdot) = \theta$.

\end{remark}

\bigskip

\section{The main isomorphism theorem}\label{main_section}

\subsection{The symplectic cotangent reduction theorems} 

We start with a few notations that will be useful in the description of the hypotheses of the theorems and in the computations that will follow:

\begin{definition}
	Let $G$ be a Lie group, acting on $\mathfrak{g}^*$, the dual of its Lie algebra, via the coadjoint action. For any $\xi \in \mathfrak{g}^*$, we denote by
	\[
	G_\xi = \{ g \in G \ | \ g \cdot \xi = 0 \}
	\]
	the \textit{stabilizer} of $\xi$, and by $\mathfrak{g}_\xi$ the Lie algebra of $G_\xi$.
\end{definition}

\begin{definition}
	Consider a twisted Hamiltonian action of $G$ on $(M, \omega, \theta)$, $\mu$ the momentum map and take $\xi \in \mathfrak{g}^*$. 
	
	Then we denote by $\xi' = \xi_{| \mathfrak{g}_\xi} \in \mathfrak{g}^*_\xi$ and by $\mu':M \to \mathfrak{g}^*_\xi$ the momentum map associated to the restricted action of $G_\xi$ on $M$.
	
	Note that $\mu^{-1}(\xi) \subset \mu'^{-1}(\xi')$.
\end{definition}

\bigskip

We recall the classical result in symplectic geometry regarding the cotangent bundle reduction (see \cite[Theorem 2.2.1]{ratiu}):

\begin{theorem}
\label{thm:embeddingRatiu}	
Let $Q$ be a manifold with a free and proper action of a Lie group $G$. Choose $\xi \in \mathfrak{g}^*$.

Let $\omega$ and $\omega_0$ be the canonical symplectic forms on $T^*Q$ and $T^*(Q/G_\xi)$ respectively. Denote by $\mu:T^*Q \to \mathfrak{g}^*$ the corresponding momentum mapping and take $\xi \in \mathfrak{g}^*$.

We have the following diagram:

\begin{figure}[H]
	\small
	\begin{tikzpicture}
		\matrix (m) [matrix of math nodes,row sep=3em,column sep=4em,minimum width=2em]
		{
			(T^*Q, \omega) & (T^*(Q/G_\xi), \omega') \\
			Q &  \\
			& Q/G_\xi \\};
		\path[-stealth]
		(m-1-1) edge node [left] {$\pi$} (m-2-1)
		(m-1-2) edge node [right] {$\pi'$} (m-3-2)
		(m-2-1) edge node [above] {$p$} (m-3-2);
	\end{tikzpicture}
\end{figure}

where $\omega$ and $\omega'$ are the symplectic forms on the respective cotangent bundles.

Assume there exists $\alpha_\xi \in \Omega^1(Q)$ such that $(\alpha_\xi)_q \in \mu'^{-1}(\xi')$ for all $q \in Q$ and which is $G_\xi$ invariant.

Then:

\begin{enumerate}
	\item[\rm{i)}] There exists a closed $2$-form $\beta_\xi$ on $Q/G_\xi$ such that $p^* \beta_\xi = d \alpha_\xi$.
	
	Let $B_\xi = \pi'^* \beta_\xi$.
	
	\item[\rm{ii)}] There is a canonical embedding of symplectic manifolds 
	\[
	\varphi: ((T^*Q)_\xi, \omega_\xi) \to (T^*(Q/G_\xi), \omega' + B_\xi),
	\]
	(where $(T^*Q)_\xi$ is the symplectic reduction of $T^*Q$) whose image covers $Q/G_\xi$.
	
	This is an isomorphism if and only if $\mathfrak{g}_\xi = \mathfrak{g}$.
	
	\item[\rm{iii)}] If, additionally, $(\alpha_\xi)_q \in \mu^{-1}(\xi)$ for all $q \in Q$, then
	\[
	\Im \varphi = \Ann (p_* \mathcal{O}),
	\]
	where $\mathcal{O}$ is the subbundle of $TQ$ tangent to the orbits of $G$.
\end{enumerate}
\end{theorem}

On the existence of a $1$-form $\alpha_\xi$ as above, the following can be proven (see again \cite[Section 2.2]{ratiu}):

\begin{proposition}
\label{prop:alphaxiSymplectic}
	\begin{enumerate}
		\item[\rm{1)}] If $\eta$ is a connection of the principal bundle $Q \to Q/G_\xi$, then \[\alpha_\xi = \xi'(\eta) \] satisfies the hypotheses in \ref*{thm:embeddingRatiu}.
		\item[\rm{2)}] If $\eta$ is a connection of the principal bundle $Q \to Q/G$, then\[ \alpha_\xi = \xi(\eta) \] satisfies the hypotheses in \ref*{thm:embeddingRatiu} and the additional condition in \rm{iii}) \ie $\alpha_\xi \in \mu^{-1}(\xi)$.
	\end{enumerate}
\end{proposition}

\subsection{The main results}

We now state and prove our theorem, which can be seen to accurately generalize \ref{thm:embeddingRatiu} to the LCS setting:

\begin{theorem}
\label{thm:embedding}	
Let $Q$ be a manifold with a free and proper action of a connected Lie group $G$. Choose $\xi \in \mathfrak{g}^*$ and take $\overline{\theta} \in \Omega^1(Q/G)$ a closed $1$-form. 

Denote by $\mu:T^*Q \to \mathfrak{g}^*$ the corresponding momentum mapping and take $\xi \in \mathfrak{g}^*$.

We have the following diagram:

\begin{figure}[H]
	\label{diag:thmEmbedding}
\begin{tikzpicture}
\matrix (m) [matrix of math nodes,row sep=3em,column sep=4em,minimum width=2em]
{
	(T^*Q, \omega_{\tilde{\theta}}, \tilde{\theta}) & (T^*(Q/G_\xi), \omega_{\tilde{\theta}'}, \tilde{\theta}') \\
	(Q, \theta) &  \\
	& (Q/G_\xi, \overline{\theta}) \\};
\path[-stealth]
(m-1-1) edge node [left] {$\pi$} (m-2-1)
(m-1-2) edge node [right] {$\pi'$} (m-3-2)
(m-2-1) edge node [above] {$p$} (m-3-2);
\end{tikzpicture}
\end{figure}

where $\theta = p^* \overline{\theta}$, $\tilde{\theta} = \pi^* \theta$, $\tilde{\theta}' = \pi'^* \overline{\theta}$ and $\omega_{\tilde{\theta}}$ and $\omega_{\tilde{\theta}'}$ are the LCS forms on the respective cotangent bundles, as in \ref{ex:cotangentLCS}.

Assume there exists $\alpha_\xi \in \Omega^1(Q)$ such that $(\alpha_\xi)_q \in \mu'^{-1}(\xi')$ for all $q \in Q$ and which satisfies
\begin{equation}
\label{eq:thmEmbedding1}
\mathcal{L}_{X_a} \alpha_\xi = \xi(a) \theta, \ \forall a \in \mathfrak{g}_\xi.
\end{equation}

Then:

\begin{enumerate}
	\item[\rm{i)}] There exists a $d_{\overline{\theta}}$-closed $2$-form $\beta_\xi$ on $Q/G_\xi$ such that $p^* \beta_\xi = d_\theta \alpha_\xi$.
	
	Let $B_\xi = \pi'^* \beta_\xi$.
	
	\item[\rm{ii)}] There is a canonical embedding of LCS manifolds 
	\[
	\varphi: ((T^*Q)_\xi, \omega_\xi, \tilde{\theta}_\xi) \to (T^*(Q/G_\xi), \omega_{\tilde{\theta}'} + B_\xi, \tilde{\theta}'),
	\]
	(where $(T^*Q)_\xi$ is the LCS reduction of $T^*Q$ performed as in \ref{th:LCSred}) whose image covers $Q/G_\xi$.
	
	This is an isomorphism if and only if $\mathfrak{g}_\xi = \mathfrak{g}$.
	
	\item[\rm{iii)}] If, additionally, $(\alpha_\xi)_q \in \mu^{-1}(\xi)$ for all $q \in Q$, then
	\[
	\Im \varphi = \Ann (p_* \mathcal{O}),
	\]
	where $\mathcal{O}$ is the subbundle of $TQ$ tangent to the orbits of $G$.
\end{enumerate}

\end{theorem}

\bigskip

\textbf{Outline of proof.} Note first that, in order to apply \ref{th:LCSred} to the above context, the LCS reduction has to make sense \ie the quotient by the foliation has to be smooth. This is not clear \textit{a priori}, but will be shown in the body of the proof.

Items i) and iii) are computational and we will go into detail on them in the body of the proof. 

The bulk of the proof concerns item ii). As it is rather technical, we begin by explaining the broad steps and ideas involved:

\begin{enumerate}
	\item[\textbf{Step 1.}] We first prove the theorem for reduction at $\xi = 0$, in which case the statement is that there exists a map 
	\[
	\varphi_0: ((T^*Q)_0 = \mu^{-1}(0)/G, \omega_0, \tilde{\theta}_0) \xrightarrow{\sim} (T^*(Q/G), \omega_{\tilde{\theta}'}, \tilde{\theta}')
	\]
	that is an LCS isomorphism. This map is canonical and can be given explicitly. 
	
	\item[\textbf{Step 2.}] We next prove the theorem for any $\xi \in \mathfrak{g}^*$, but not yet in full generality, assuming the stabilizer of $\xi$ with respect to the coadjoint action is the whole group: $G_\xi = G$. We then again want to prove the existence of an LCS isomorphism  
	\[
	\varphi: ((T^*Q)_\xi, \omega_\xi, \tilde{\theta}_\xi) \to (T^*(Q/G), \omega_{\tilde{\theta}'} + B_\xi, \tilde{\theta}').
	\]
	
	We use the form $\alpha_\xi$ required in the hypothesis to construct a ``shift" diffeomorphism $\overline{S}_\xi$ between the $\xi$-level set and the $0$-level set of the momentum map and then show that it descends to the quotients \ie the reduced spaces at $\xi$ and $0$, respectively. Explicitly, this diffeomorphism is just subtracting $\alpha_\xi$ on every fiber.
	
	Composing with the LCS isomorphism $\varphi_0$ described above, we obtain the required map $\varphi$; the change of the LCS form on $T^*(Q/G)$ to $\omega_{\tilde{\theta}'} + B_\xi$ is determined by this usage of $\alpha_\xi$. 
	
	This is all depicted in the diagram below:
	
	\begin{figure}[H]
		\label{diag:thmEmbeddingOutline}
		\begin{tikzpicture}
		\matrix (m) [matrix of math nodes,row sep=3em,column sep=4em,minimum width=2em]
		{
			&  T^*Q & \\
			\mu^{-1}(\xi) & \mu^{-1}(0) &  \\
			(T^*Q)_\xi & (T^*Q)_0 & T^*(Q/G) \\};
		\path[-stealth]
		(m-2-1) edge node [above] {$i$} (m-1-2)
		(m-2-2) edge node [left] {$i$} (m-1-2)
		(m-2-1) edge node [above] {$\overline{S}_\xi$} (m-2-2)
		(m-2-1) edge node [below] {$\sim$} (m-2-2)
		(m-3-1) edge node [above] {$S_\xi$} (m-3-2)
		(m-3-1) edge node [below] {$\sim$} (m-3-2)
		(m-2-2) edge node [left] {$\pi_0$} (m-3-2)
		(m-3-2) edge node [above] {$\varphi_0$} (m-3-3)
		(m-3-2) edge node [below] {$\sim$} (m-3-3)
		(m-2-1) edge node [left] {$\pi_\xi$} (m-3-1)
		(m-3-1) edge[bend right] node [above] {$\varphi$} (m-3-3)
		(m-3-1) edge[bend right] node [below] {$\sim$} (m-3-3);
		\end{tikzpicture}
	\end{figure}

	\item[\textbf{Step 3.}] In full generality, we consider only the action of $G_\xi$ on $T^*Q$ and notice that, since $(G_\xi)_{\xi'} = G_\xi$, we can apply the previous step for this action, with respect to the value $\xi'$, to obtain an LCS isomorphism
	\[
	\varphi': ((T^*Q)_{\xi'}, \omega_{\xi'}, \tilde{\theta}_{\xi'}) \to (T^*(Q/G_\xi), \omega_{\tilde{\theta}'} + B_\xi, \tilde{\theta}').
	\]
	We then use the inclusion $(T^*Q)_\xi \subset (T^*Q)_{\xi'}$ to construct the symplectic embedding.
\end{enumerate}

\bigskip

\begin{proof}[Proof of \ref{thm:embedding}]
	We begin by proving i). Since $\theta(X_a) = 0$ and using \eqref{eq:thmEmbedding1},
	\begin{equation*}
	\begin{split}
	\xi(a) \theta &= \mathcal{L}_{X_a} \alpha_\xi = \mathcal{L}_{X_a}^\theta \alpha_\xi = d_\theta \alpha_\xi(X_a) + i_{X_a} d_\theta \alpha_\xi = d_\theta (-\xi(a)) + i_{X_a} d_\theta \alpha_\xi\\
	&= \xi(a)\theta + i_{X_a} d_\theta \alpha_\xi, 
	\end{split}
	\end{equation*}
	so
	\[
	i_{X_a} d_\theta \alpha_\xi = 0, \ \forall a \in \mathfrak{g}_\xi.
	\]
	This also implies
	\[
	\mathcal{L}_{X_a} d_\theta \alpha_\xi = \mathcal{L}^\theta_{X_a} d_\theta \alpha_\xi = d_\theta (i_{X_a} d_\theta \alpha_\xi) = 0,
	\]
	so $d_\theta \alpha_\xi$ descends to a $d_{\overline{\theta}}$-closed (but not necessarily exact) $2$-form $\beta_\xi$ on $Q/G_\xi$. 
	
	\bigskip
	
	We not turn to item ii), following the three steps enumerated in the outline.
	
	\bigskip
	
	\textbf{Step 1.} Assume $\xi = 0$.
	
	The reduction at zero is just $(T^*Q)_0 = \mu^{-1}(0) / G$. Recall from \ref{prop:momentumCTG} that 
	\[
	\mu^{-1}(0) = \{ \alpha \in T^*Q \ | \ \alpha(X_a) = 0, \ \forall a \in \mathfrak{g} \}.
	\]
	
	Define $\overline{\varphi}_0: \mu^{-1}(0) \to T^*(Q/G)$ by
	\[
	\overline{\varphi}_0(\alpha_q)(v_{\hat{q}}) = \alpha_q (v_q), \text{ for any } \alpha_q \in \mu^{-1}(0) \text{ and } v_{\hat{q}} \in T_{\hat{q}} (Q/G),
	\]
	where $v_q \in T_q Q$ such that $p_* v_q = v_{\hat{q}}$.
	
	We first prove that $\overline{\varphi}_0$ is well defined. If $v_q$ and $w_q$ satisfy $p_* v_q = p_* w_q = v_{\hat{q}}$, then $v_q = w_q + (X_a)_q$ for some $a \in \mathfrak{g}$, so 
	\[
	\alpha_q(v_q) = \alpha_q (w_q + (X_a)_q) = \alpha_q(w_q),
	\]
	since $\alpha_q \in \mu^{-1}(0)$.
	
	Secondly, $\overline{\varphi}_0$ is surjective. Indeed, for $\eta \in T_{\hat{q}}^*(Q/G)$, take $\alpha_q = p^* \eta \in T_q^*Q$; then, by definition, $\overline{\varphi}_0(\alpha_q) = \eta$.
	
	Thirdly, notice that $\overline{\varphi}_0$ is $G$-invariant: take an $\alpha_q \in \mu^{-1}(0)$ and $g \in G$. If $v_q \in T_q Q$ satisfies $p_* v_q = v_{\hat{q}}$, then $p_* (g_* v_q) = v_{\hat{q}}$, so, following the definitions of $\overline{\varphi}_0$ and the action of $G$ on $T^*Q$,
	\[
	\overline{\varphi}_0(g \cdot \alpha_q)(v_{\hat{q}}) = (g \cdot \alpha_q) (g_* v_q) = \alpha_q(g^{-1}_* g_* v_q) = \alpha_q (v_q) = \overline{\varphi}_0(\alpha_q)(v_q).
	\]
	
	It follows that $\overline{\varphi}_0$ descends to a surjective map on the quotient 
	\[
	\varphi_0: \mu^{-1}(0) / G \to T^*(Q/G).
	\]
	We now show that this is also injective. 
	
	If $\varphi_0([\alpha_q]) = \varphi_0([\gamma_{q'}])$, then $\overline{\varphi}_0(\alpha_q) = \overline{\varphi}_0(\gamma_{q'})$, so there is a $g \in G$ such that $q' = g \cdot q$. Then
	\begin{equation*}
	\begin{split}
	(g \cdot \alpha_q)(v_{g \cdot q}) &= \alpha_q(g^{-1}_* v_{g \cdot q}) = \overline{\varphi}_0(\alpha_q)(p_* g^{-1}_* v_{g \cdot q}) = \overline{\varphi}_0(\alpha_q)(p_*  v_{g \cdot q}) \\
	&= \overline{\varphi}_0(\gamma_{g \cdot q})(p_* v_{g \cdot q}) = \gamma_{g \cdot q}(v_{g \cdot q}), \ \forall \ v_{g \cdot q} \in T_{g \cdot q}Q,
	\end{split}
	\end{equation*}
	\ie $g \cdot \alpha_q = \gamma_{q'}$, so $[\alpha_q] = [\gamma_{q'}]$.
	
	Hence, $\varphi_0$ is a bijection from $(T^*Q)_0$ to $T^*(Q/G)$. We next show that $\varphi_0^* \omega_{\tilde{\theta}'} = \omega_0$, where $\omega_0$ is the LCS form on the reduced manifold $(T^*Q)_0$. Keeping in mind the definition of $\omega_0$, it is sufficient to show that $\overline{\varphi}_0^* \omega_{\tilde{\theta}'} = \omega_{\tilde{\theta}}$. But as $\omega_{\tilde{\theta}'} = d_{\tilde{\theta}'} \eta'$ and $\omega_{\tilde{\theta}} = d_{\tilde{\theta}} \eta$, where $\eta$ and $\eta'$ are the tautological $1$-forms on their respective cotangent bundles, it is furthermore sufficient to show that $\overline{\varphi}_0^* \tilde{\theta}' = \tilde{\theta}$ and $\overline{\varphi}_0^* \eta' = \eta$.
	
	Indeed, 
	\begin{equation}
	\label{eq:thmEmbedding7}
	\overline{\varphi}_0^* \tilde{\theta}' = \overline{\varphi}_0^* \pi'^* \overline{\theta} = \pi^* p^* \overline{\theta} = \tilde{\theta},
	\end{equation}
	as $\pi' \circ \overline{\varphi}_0 = p \circ \pi$, and, using this same equality,
	\begin{equation}
	\label{eq:thmEmbedding8}
	\begin{split}
	(\overline{\varphi}_0^* \eta')_{\alpha_q}(V_{\alpha_q}) &= \eta'_{\overline{\varphi}_0(\alpha_q)}(\overline{\varphi}_{0*} V_{\alpha_q}) = \overline{\varphi}_0(\alpha_q)(\pi'_* \overline{\varphi}_{0*} V_{\alpha_q}) = \overline{\varphi_0}(\alpha_q)(p_* \pi_* V_{\alpha_q}) \\
	&= \alpha_q (\pi_* V_{\alpha_q}) = \eta_{\alpha_q}(V_{\alpha_q}).
	\end{split}
	\end{equation}
	
	It follows, since the LCS forms are non-degenerate, that $\varphi_0$ is a bijective immersion, hence an diffeomorphism preserving the LCS structure.
	
	\bigskip
	
	\textbf{Step 2.} Assume $G = G_\xi$. Notice that, in this case, we ask that $\alpha_\xi \in \mu^{-1}(\xi)$ and
	\[
	\mathcal{L}_{X_a} \alpha_\xi = \xi(a) \theta, \ \forall a \in \mathfrak{g}.
	\]
	
	Define a ``shifting map" between level sets of the momentum map:
	\[
	\overline{S}_\xi : \mu^{-1}(\xi) \to \mu^{-1}(0), \ \overline{S}_\xi (\alpha_q) = \alpha_q - (\alpha_\xi)_q
	\]
	(this is well defined as the momentum map is linear on the fibers). 
	
	Notice $\overline{S}_\xi$ is a diffeomorphism.
	
	We first calculate the pullback of the LCS structure on $T^*Q$ through $\overline{S}_\xi$. Since $\pi \circ \overline{S}_\xi = \pi$,
	\begin{equation}
	\label{eq:thmEmbedding2}
	\overline{S}_\xi^* \tilde{\theta} = \overline{S}_\xi^* \pi^* \theta = \pi^* \theta = \tilde{\theta}.
	\end{equation}
	
	On the other hand,
	\begin{equation*}
	\begin{split}	
	(\overline{S}_\xi^* \eta)_{\alpha_q}(V_{\alpha_q}) &= \eta_{\alpha_q - (\alpha_\xi)_q}(\overline{S}_{\xi*} V_{\alpha_q}) = (\alpha_q - (\alpha_\xi)_q) (\pi_* \overline{S}_{\xi*} V_{\alpha_q}) \\
	&= \alpha_q (\pi_* V_{\alpha_q}) - (\alpha_\xi)_q (\pi_* V_{\alpha_q}) \\
	&= \eta_{\alpha_q}(V_{\alpha_q}) - (\pi^* \alpha_\xi)_{\alpha_q} (V_{\alpha_q}),
	\end{split}
	\end{equation*}
	\ie
	\begin{equation}
	\label{eq:thmEmbedding3}
	\overline{S}_\xi^* \eta = \eta - \pi^* \alpha_\xi.
	\end{equation}
	Apply $d_{\tilde{\theta}}$ to \eqref{eq:thmEmbedding3} and use \eqref{eq:thmEmbedding2} to find 
	\begin{equation}
	\label{eq:thmEmbedding4}
	\overline{S}_\xi^* \omega_{\tilde{\theta}} = \omega_{\tilde{\theta}} - \pi^* d_\theta \alpha_\xi.
	\end{equation}
	We will need this computation for the following 
	
	\textbf{Claim.} $\overline{S}_\xi$ transports the distribution (see \ref{rmk:explicatiiLCSred})
	\[
	\mathcal{F}_\xi = T\mu^{-1}(\xi) \cap (T\mu^{-1}(\xi))^{\omega_{\tilde{\theta}}} = \{ \tilde{X}_a + \xi(a)\tilde{\theta}^{\omega_{\tilde{\theta}}} \ | \ a \in \mathfrak{g}_\xi = \mathfrak{g} \}
	\]
	on $\mu^{-1}(\xi)$ to the distribution 
	\[
	\mathcal{F}_0 = T\mu^{-1}(0) \cap (T\mu^{-1}(0))^{\omega_{\tilde{\theta}}} = \{ \tilde{X}_a \ | \ a \in \mathfrak{g} \}
	\]
	on $\mu^{-1}(0)$, hence descends to a map $S_\xi: (T^*Q)_\xi \to (T^*Q)_0$. In particular, $(T^*Q)_\xi$ is also smooth.
	
	\bigskip
	
	To prove this claim, first notice that $\tilde{\theta}^{\omega_{\tilde{\theta}}}$ is tangent to the fibers of $T^*Q$. Indeed, for any such fiber $F$,
	\[
	F \subset \Ker \tilde{\theta} \implies F \supset (\Ker \tilde{\theta})^{\omega_{\tilde{\theta}}} = \langle \tilde{\theta}^{\omega_{\tilde{\theta}}} \rangle,
	\]
	since $F$ is a Lagrangian submanifold of $T^*Q$.
	
	Now, for an $\alpha_q \in \mu^{-1}(\xi)$ and $V \in T_{\alpha_q - (\alpha_\xi)_q}(T^*Q)$, using this fact and \eqref{eq:thmEmbedding4}, we have
	\begin{equation*}
	\begin{split}	
	(\omega_{\tilde{\theta}})_{\alpha_q - (\alpha_\xi)_q} (\overline{S}_{\xi*} \tilde{\theta}_{\alpha_q}^{\omega_{\tilde{\theta}}}, V) &= (\omega_{\tilde{\theta}})_{\alpha_q - (\alpha_\xi)_q} (\overline{S}_{\xi*} \tilde{\theta}_{\alpha_q}^{\omega_{\tilde{\theta}}}, \overline{S}_{\xi*}(\overline{S}^{-1}_\xi)_*V) \\
	&= (\overline{S}_\xi^* \omega_{\tilde{\theta}})_{\alpha_q}(\tilde{\theta}_{\alpha_q}^{\omega_{\tilde{\theta}}}, (\overline{S}^{-1}_\xi)_*V) \\ 
	&= (\omega_{\tilde{\theta}} - \pi^* d_\theta \alpha_\xi)_{\alpha_q} (\tilde{\theta}_{\alpha_q}^{\omega_{\tilde{\theta}}}, (\overline{S}^{-1}_\xi)_*V) \\
	&= \tilde{\theta}_{\alpha_q}((\overline{S}^{-1}_\xi)_*V) - (\pi^* d_\theta \alpha_\xi)_{\alpha_q} (\tilde{\theta}_{\alpha_q}^{\omega_{\tilde{\theta}}}, (\overline{S}^{-1}_\xi)_*V) \\
	&= \theta_q(\pi_* (\overline{S}^{-1}_\xi)_*V) = \theta_q(\pi_* V) \\
	&= \tilde{\theta}_{\alpha_q - (\alpha_\xi)_q} (V)
	\end{split}
	\end{equation*}
	hence
	\begin{equation}
	\label{eq:thmEmbedding5}
	\overline{S}_{\xi*} \tilde{\theta}^{\omega_{\tilde{\theta}}} = \tilde{\theta}^{\omega_{\tilde{\theta}}}.
	\end{equation}
	
	The claim is then equivalent to 
	\[
	\overline{S}_{\xi*}(\tilde{X}_a) = \tilde{X}_a - \xi(a) \tilde{\theta}^{\omega_{\tilde{\theta}}}, \ \forall a \in \mathfrak{g}.
	\]
	
	\textbf{\ref{lem:thmEmbedding}}
	\textit{The condition
		\[
		\overline{S}_{\xi*}(\tilde{X}_a) = \tilde{X}_a - \xi(a) \tilde{\theta}^{\omega_{\tilde{\theta}}}, \ \forall a \in \mathfrak{g}.
		\]
		is equivalent to assumption \eqref{eq:thmEmbedding1} in the hypothesis.}

	We postpone the proof of this lemma until after the main theorem is proven.
	
	\bigskip
	
	The mapping $\overline{S}_\xi$ then descends to a diffeomorphism
	\[
	S_\xi: (T^*Q)_\xi \to (T^*Q)_0.
	\]
	Define
	\[
	\varphi = \varphi_0 \circ S_\xi : (T^*Q)_\xi \to T^*(Q/G),
	\]
	where $\varphi_0$ is the diffeomorphism defined in Step 1.
	
	It now only remains to show that 
	\[
	\varphi^* (\omega_{\tilde{\theta}'} + B_\xi) = \omega_\xi \text{ and } \varphi^* \tilde{\theta}' = \tilde{\theta}_\xi.
	\]
	Since $\omega_\xi$ and $\tilde{\theta}_\xi$ are defined by the equations $p_\xi^* \omega_\xi = \omega_{\tilde{\theta}}$ and $p_\xi^* \theta_\xi = \tilde{\theta}$, where $p_\xi:\mu^{-1}(\xi) \to (T^*Q)_\xi$ is the natural projection, this is equivalent to
	\[
	\overline{S}_\xi^* \overline{\varphi}_0^* (\omega_{\tilde{\theta}'} + B_\xi) = \omega_{\tilde{\theta}} \text{ and } \overline{S}_\xi^* \overline{\varphi}_0^* \tilde{\theta}' = \tilde{\theta}.
	\]
	But, using the definition of $B_\xi$, \eqref{eq:thmEmbedding8} and \eqref{eq:thmEmbedding4},
	\begin{equation*}
	\begin{split}
	\overline{S}_\xi^* \overline{\varphi}_0^* (\omega_{\tilde{\theta}'} + B_\xi) &= \overline{S}_\xi^* \overline{\varphi}_0^* (\omega_{\tilde{\theta}'} + \pi'^* \beta_\xi) = \overline{S}_\xi^* (\omega_{\tilde{\theta}} + \pi^* p^* \beta_\xi) \\
	&= \omega_{\tilde{\theta}} - \pi^* d_\theta \alpha_\xi + \pi^* d_\theta \alpha_\xi \\
	&= \omega_{\tilde{\theta}}.
	\end{split}
	\end{equation*}
	
	Similarly, using \eqref{eq:thmEmbedding7} and \eqref{eq:thmEmbedding2},
	\begin{equation*}
	\begin{split}
	\overline{S}_\xi^* \overline{\varphi}_0^* \ \tilde{\theta}' &= \overline{S}_\xi^* \tilde{\theta} = \tilde{\theta}.
	\end{split}
	\end{equation*}
	
	\bigskip
	
	\textbf{Step 3.} In full generality, having $(\alpha_\xi)_q \in \mu'^{-1}(\xi')$, for all $q \in Q$, we can apply Step 2 for the action of $G_\xi$ on $Q$, since, obviously, $(G_\xi)_{\xi'} = G_\xi$.
	
	This produces a diffeomorphism compatible with the LCS structures
	\[
	\varphi' : (\mu'^{-1}(\xi')/\mathcal{F}_\xi, \omega_{\xi'}, \tilde{\theta}_{\xi'}) \to (T^*(Q/G_\xi), \omega_{\tilde{\theta}'} + B_\xi, \tilde{\theta}'),
	\]
	where 
	\begin{equation*}
	\begin{split}
	\mathcal{F}_\xi &= T\mu^{-1}(\xi) \cap (T\mu^{-1}(\xi))^{\omega_{\tilde{\theta}}} = \{ \tilde{X}_a + \xi(a)\tilde{\theta}^{\omega_{\tilde{\theta}}} \ | \ a \in \mathfrak{g}_\xi \} \\
	&= T\mu'^{-1}(\xi') \cap (T\mu'^{-1}(\xi'))^{\omega_{\tilde{\theta}}}.
	\end{split}
	\end{equation*}
	
	But note that 
	\[
	\mu^{-1}(\xi) \subset \mu'^{-1}(\xi') \implies (\mu^{-1}(\xi)/\mathcal{F}_\xi) \subset (\mu'^{-1}(\xi')/\mathcal{F}_\xi).
	\]
	Let $i: (\mu^{-1}(\xi)/\mathcal{F}_\xi) \hookrightarrow (\mu'^{-1}(\xi')/\mathcal{F}_\xi)$ be this inclusion. It is clear that $i$ is compatible with the LCS structure on the two reduced spaces
	\[
	i^* \omega_{\xi'} = \omega_\xi \text{ and } i^* \theta_{\xi'} = \theta_\xi.
	\] 
	so, in particular, it is an embedding.
	
	Moreover, we claim $(\mu^{-1}(\xi)/\mathcal{F}_\xi) = (\mu'^{-1}(\xi')/\mathcal{F}_\xi)$ if and only if $\mathfrak{g} = \mathfrak{g_\xi}$. If $\mathfrak{g} = \mathfrak{g_\xi}$, the equality is clear by definition. 
	
	If $(\mu^{-1}(\xi)/\mathcal{F}_\xi) = (\mu'^{-1}(\xi')/\mathcal{F}_\xi)$, we also have an equality of the level sets: $\mu^{-1}(\xi) = \mu'^{-1}(\xi')$. But recall from \ref{rmk:explicatiiLCSred} that 
	\[
	(T\mu^{-1}(\xi))^{\omega_{\tilde{\theta}}} = \{ \tilde{X}_a + \xi(a)\tilde{\theta}^{\omega_{\tilde{\theta}}} \ | \ a \in \mathfrak{g}\}
	\]
	while 
	\[
	(T\mu'^{-1}(\xi'))^{\omega_{\tilde{\theta}}} = \{ \tilde{X}_a + \xi(a)\tilde{\theta}^{\omega_{\tilde{\theta}}} \ | \ a \in \mathfrak{g}_\xi\}.
	\] 
	Together with the fact that $G$ acts freely on $Q$, so $\tilde{X}_a = \tilde{X}_b \implies a = b$ for any $a, b \in \mathfrak{g}$, this implies $\mathfrak{g} = \mathfrak{g_\xi}$.
	
	Thus, taking $\varphi = \varphi' \circ i$, we have the required embedding compatible with the LCS structures, which is an diffeormorphism if and only if $\mathfrak{g} = \mathfrak{g_\xi}$.
	
	\bigskip 
	
	We end with the proof of iii) \ie $\Im \varphi = \Ann (p_* \mathcal{O})$. Note first that 
	\[
	\Ker p_{* | \mathcal{O}} = \{ X_a \ | \ a \in \mathfrak{g}_\xi \}
	\]
	is of constant dimension equal to $\dim G_\xi$, as the action is free, so $p_* \mathcal{O}$ is also of constant dimension and so is it's annihilator. This shows the right hand side is a subbundle of $T^*(Q/G_\xi)$.
	
	Take $\alpha_q \in \mu^{-1}(\xi)$ and $a \in \mathfrak{g}$. Then
	\begin{equation*}
	\begin{split}
	\varphi([\alpha_q])(p_* X_a) &= \varphi_0 S_\xi([\alpha_q])(p_* X_a) = \overline{\varphi}_0 \overline{S}_\xi (\alpha_q)(p_* X_a) \\
	&= (\alpha_q - \alpha_\xi(q))(X_a) = -\xi(a) + \xi(a) \\
	&= 0.
	\end{split}
	\end{equation*}
	
	For the other inclusion, let $\eta_{\hat{q}} \in \Ann (p_* \mathcal{O})$ and choose 
	\[
	\alpha_q = p^* \eta_{\hat{q}} + \alpha_\xi(q).
	\]
	Then, for any $a \in \mathfrak{g}$,
	\[
	\alpha_q(X_a) = \eta_{\hat{q}}(p_* X_a) - \xi(a) = -\xi(a),
	\]
	so $\alpha_q \in \mu^{-1}(\xi)$. In addition, for any $Y$ vector field on $Q$, by the definition of $\alpha_q$,
	\[
	\varphi([\alpha_q])(p_* Y) = \overline{\varphi}_0 \overline{S}_\xi (\alpha_q)(p_* Y) = p^* \eta_{\hat{q}} (Y) = \eta_{\hat{q}}(p_* Y),
	\]
	hence $\varphi([\alpha_q]) = \eta_{\hat{q}}$.
	This concludes the proof.
\end{proof}
	
	\bigskip
	
It only remains to give the proof of:
\begin{lemma}
	\label{lem:thmEmbedding} In the context and with the notations of the previous theorem, for an $a \in \mathfrak{g}$, the following are equivalent
	\begin{enumerate}
		\item[\rm{(1)}] $\overline{S}_{\xi*}(\tilde{X}_a) = \tilde{X}_a - \xi(a) \tilde{\theta}^{\omega_{\tilde{\theta}}}; $
		\item[\rm{(2)}] $\mathcal{L}_{X_a} \alpha_\xi = \xi(a) \theta.$
	\end{enumerate}

\begin{proof}
	
	We make the necessary calculations in local coordinates. 
	
	Let $(U, (q_i)_{i=\overline{1,n}}) \subset Q$ be an open subset together with a chart, so $T^*Q_{|U} = T^* U \simeq U \times \mathbb{R}^n$. Denote by $(q_i, c_i)_{i=\overline{1,n}}$ the coordinates on $T^*U$. For any $a \in \mathfrak{g}$ and any $q \in U$, let
	\[
	(X_a)_q = \sum_{i=1}^{n} X^i_a(q) \frac{\partial}{\partial q_i}.
	\]
	We first derive the expression of $\tilde{X}_a$. Since $\pi_* \tilde{X}_a = X_a$, we have, for any $\alpha_q \in T^*U$, $\alpha_q = \sum_{i=1}^{n} \alpha_q^i dq_i$,
	\[
	(\tilde{X}_a)_{\alpha_q} = \sum_{i=1}^{n} X^i_a(q) \frac{\partial}{\partial q_i} + \sum_{i=1}^{n} b^i_a(q) \frac{\partial}{\partial c_i},
	\]
	where, by the definition of the $G$ action on $T^*Q$ as push-forward, the components
	\begin{equation*}
	\begin{split}
	b^k_a \left( q \right)   &=  \left( \frac{d}{dt}_{|{t=0}}  \left(   \left( e^{-ta} \right)  ^* \alpha_q \right)   \right)   \left( \frac{\partial}{\partial c_k} \right)   =  \left( \frac{d}{dt}_{|{t=0}}  \left(  \sum_{i=1}^{n} \alpha_q^i d  \left( e^{-ta} \cdot q \right)  _i  \right)   \right)   \left( \frac{\partial}{\partial c_k} \right)   \\ 
	&=  \left( \frac{d}{dt}_{|{t=0}}  \left(  \sum_{i=1}^{n} \alpha_q^i \sum_{j=1}^{n} \frac{\partial  \left( e^{-ta} \right)  _i}{\partial q_j} dq_j  \right)   \right)   \left( \frac{\partial}{\partial c_k} \right)   \\
	&= \frac{d}{dt}_{|{t=0}}  \left(  \sum_{i=1}^{n} \alpha_q^i \frac{\partial  \left( e^{-ta} \right)  _i}{\partial q_k}  \right)   = \sum_{i=1}^{n} \alpha_q^i \frac{\partial  \left( \frac{d}{dt}_{|{t=0}}  \left( e^{-ta} \right)  _i \right)  }{\partial q_k}  \\
	&= -\sum_{i=1}^{n} \alpha_q^i \frac{\partial X_a^i \left( q \right)  }{\partial q_k},
	\end{split}
	\end{equation*}
	hence
	\begin{equation}
	\label{eq:lem:thmEmbedding1}
	 \left( \tilde{X}_a \right)  _{\alpha_q} = \sum_{i=1}^{n} X^i_a \left( q \right)   \frac{\partial}{\partial q_i} - \sum_{i=1}^{n}  \left( \sum_{j=1}^{n} \alpha_q^j \frac{\partial X_a^j \left( q \right)  }{\partial q_i} \right)   \frac{\partial}{\partial c_i}.
	\end{equation}
	
	Let $\alpha_\xi(q) = \sum_{i=1}^n \alpha_\xi^i(q) dq_i$. By definition and using \eqref{eq:lem:thmEmbedding1},
	\begin{equation}
	\label{eq:lem:thmEmbedding2}
	\begin{split}
	\overline{S}_{\xi*} ((\tilde{X}_a )_{\alpha_q} )   &= \frac{d}{dt}_{|{t=0}}  \left( \overline{S}_{\xi} \left( e^{ta} \cdot \alpha_q \right)   \right)   = \frac{d}{dt}_{|{t=0}}  \left( e^{ta} \cdot \alpha_q - \alpha_\xi \left( e^{ta} \cdot q \right)   \right)   \\
	&= \sum_{i=1}^{n} X^i_a \left( q \right)   \frac{\partial}{\partial q_i} - \sum_{i=1}^{n}  \left( \sum_{j=1}^{n} \alpha_q^j \frac{\partial X_a^j \left( q \right)  }{\partial q_i} \right)   \frac{\partial}{\partial c_i} \\
	&- \sum_{i=1}^n \frac{d}{dt}_{|{t=0}} \left( \alpha_\xi^i \left( e^{ta} \cdot q \right)   \right)   \frac{\partial}{\partial c_i} \\
	&= \sum_{i=1}^{n} X^i_a \left( q \right)   \frac{\partial}{\partial q_i} - \sum_{i=1}^{n}  \left( \sum_{j=1}^{n}  \left( \alpha_q^j - \alpha_\xi^j \left( q \right)   \right)   \frac{\partial X_a^j \left( q \right)  }{\partial q_i} \right)   \frac{\partial}{\partial c_i} \\
	&- \sum_{i=1}^{n}  \left( \sum_{j=1}^{n} \alpha_\xi^j \left( q \right)   \frac{\partial X_a^j \left( q \right)  }{\partial q_i} \right)   \frac{\partial}{\partial c_i} - \sum_{i=1}^n X_a \left( \alpha_\xi^i \right)   \frac{\partial}{\partial c_i} \\
	&=  \left( \tilde{X}_a \right)  _{\alpha_q - \alpha_\xi \left( q \right)  } - \sum_{i=1}^{n}  \left( \sum_{j=1}^{n} \alpha_\xi^j \left( q \right)   \frac{\partial X_a^j \left( q \right)  }{\partial q_i} +  X_a \left( \alpha_\xi^i \right)    \right)   \frac{\partial}{\partial c_i}.
	\end{split}
	\end{equation}
	
	Let $\theta = \sum_{i=1}^n \theta_i dq_i$. The local expression of the tautological $1$-form $\eta$ on $T^*Q$ is 
	\[
	\eta = \sum_{i=1}^n c_i dq_i.
	\]
	Consequently, 
	\[
	\omega_{\tilde{\theta}} = d\eta - \theta \wedge \eta = \sum_{i=1}^n dc_i \wedge dq_i - \sum_{i<j}(\theta_i c_j - \theta_j c_i)dq_i \wedge dq_j.
	\]
	
	We have $\tilde{\theta}^{\omega_{\tilde{\theta}}} = \sum_{i=1}^n l_i \frac{\partial}{\partial q_i} + \sum_{i=1}^n p_i \frac{\partial}{\partial c_i}$, satisfying
	\[
	\omega_{\tilde{\theta}}(\tilde{\theta}^{\omega_{\tilde{\theta}}}, V) = \tilde{\theta}(V), \ \forall V \in T(T^*U)
	\]
	which, in local coordinates, is equivalent to
	\[
	\sum_{i=1}^n p_i v'_i - l_i v''_i - \sum_{i<j}(\theta_i c_j - \theta_j c_i)(l_i v'_j - l_j v'_i) = \sum_{i=1}^n \theta_i v'_i, \ \forall v'_i, v''_i \in \mathbb{R}
	\]
	\ie $ l_i = 0 \text{ and } p_i = \theta_i$ for all $i=\overline{1,n}$, therefore
	\begin{equation}
	\label{eq:lem:thmEmbedding3}
	\tilde{\theta}^{\omega_{\tilde{\theta}}} = \sum_{i=1}^n \theta_i \frac{\partial}{\partial c_i}
	\end{equation}
	(this again proves that $\tilde{\theta}^{\omega_{\tilde{\theta}}}$ is tangent to the fibers of $T^*Q$, as shown in the previous theorem).
	
	Now, 
	\begin{equation}
	\label{eq:lem:thmEmbedding4}
	\begin{split}
	\mathcal{L}_{X_a} \alpha_\xi &= \mathcal{L}_{X_a} (\sum_{i=1}^n \alpha_\xi^i dq_i) = \sum_{i=1}^n ((\mathcal{L}_{X_a} \alpha_\xi^i)dq_i + \alpha_\xi^i d (\mathcal{L}_{X_a} q_i)) \\
	&= \sum_{i=1}^n (X_a(\alpha_\xi^i)dq_i + \alpha_\xi^i d (X_a^i)) \\
	&= \sum_{i=1}^n (\alpha_\xi^i \sum_{j=1}^n \frac{\partial X_a^i}{\partial q_j} dq_j +  X_a(\alpha_\xi^i)dq_i) \\
	&= \sum_{i=1}^n (\sum_{j=1}^n \alpha_\xi^j \frac{\partial X_a^j}{\partial q_i} + X_a(\alpha_\xi^i))dq_i.
	\end{split}
	\end{equation}
	
	Comparing \eqref{eq:lem:thmEmbedding2}, \eqref{eq:lem:thmEmbedding3} and \eqref{eq:lem:thmEmbedding4}, the proof is complete.
\end{proof}
	
\end{lemma}

\section{Examples}

We now turn to a few classes of examples.

As opposed to the symplectic case (see \ref{prop:alphaxiSymplectic}), the conditions of \ref{thm:embedding} for the existence of $\alpha_\xi$ are not always satisfied. We will see this by constructing a class of examples of cotangent bundles whose reductions at non-zero regular values are not (generally) cotangent bundles, so on such manifolds no $\alpha_\xi$ with the required properties can exist.

We first prove an elementary lemma:

\begin{lemma}
\label{lem:RxM}
	Let $M$ be a manifold with a free $S^1$ action. Take $\xi \in \mathbb{R}, \xi \neq 0$ and consider the action of $\mathbb{R}$ on $\mathbb{R} \times M$,
	\[
	s \cdot (t, m) = (t + \xi s, e^{is} \cdot m), \ \forall (t, m) \in \mathbb{R} \times M, \ \forall s \in \mathbb{R}.
	\]
	
	Then there is a diffeomorphism $\varphi: (\mathbb{R} \times M)/\mathbb{R}  \xlongrightarrow{\sim} M$.
	
	If $\alpha \in \Omega^k(\mathbb{R} \times M)$ is invariant and basic with regards to this $\mathbb{R}$-action, then it descends to $\alpha_0 = \alpha_{| \{0\} \times M}$ on $M$.
	
\begin{proof}
	Consider the map
	\[
	\tilde{\varphi}: \mathbb{R} \times M \to M, \ \tilde{\varphi}(t,m) = e^{-\frac{it}{\xi}} \cdot m.
	\]
	This is obviously a surjective submersion and only depends on the equivalence class $[(t, m)]$ since, for any $s \in \mathbb{R}$,
	\begin{equation*}
	\begin{split}
	\tilde{\varphi}(t+s\xi, e^{is} \cdot m) &= e^{-\frac{i(t+s\xi)}{\xi}} \cdot (e^{is} \cdot m) = e^{-\frac{it}{\xi}} e^{-is} \cdot (e^{is} \cdot m) = e^{-\frac{it}{\xi}} \cdot m \\
	&= \tilde{\varphi}(t, m),
	\end{split}
	\end{equation*}
	so it descends to a diffeomorphism $\varphi: (\mathbb{R} \times M)/\mathbb{R} \to M$.
	
	Take now $\alpha \in \Omega^k(\mathbb{R} \times M)$ invariant and basic with regards to the $\mathbb{R}$-action. We want to show that $\tilde{\varphi}^* \alpha_0 = \alpha$. Indeed, denoting $i:\{0\} \times M \to \mathbb{R} \times M$ the natural inclusion,
	\[
	\tilde{\varphi}^* \alpha_0 = \tilde{\varphi}^* i^* \alpha = (i \circ \tilde{\varphi})^* \alpha,
	\]
	so we need $(i \circ \tilde{\varphi})^* \alpha = \alpha$. In order not to burden the notations, we prove this for $\alpha$ a $1$-form; the general case follows in exactly the same manner.
	
	Let $(u, v) \in T_t\R \times T_mM$. Note first that
	\begin{equation*}
		\begin{split}
	d_{(t, m)}\tilde{\varphi}(u, v) &=  \frac{d}{ds}_{|s=0} \tilde{\varphi}\left( t+su,  \gamma(s)\right) = \frac{d}{ds}_{|s=0} \left( e^{-\frac{i(t+su)}{\xi}} \cdot \gamma(s) \right) \\ &= \left( e^{-\frac{it}{\xi}} \cdot \right)_* \left( \frac{d}{ds}_{|s=0} \left( e^{-\frac{isu}{\xi}} \cdot \gamma(s) \right) \right) \\ &= \left( e^{-\frac{it}{\xi}} \cdot \right)_* \left( X_{-\frac{u}{\xi}}(m) + v \right),
		\end{split}
	\end{equation*}
	where $\gamma$ is an integral curve for $v$ in $m$ and $X_\cdot$ are the fundamental vector fields associated with the $S^1$-action on $M$. On the other hand, we have
	\[
	(i \circ \tilde{\varphi})(t, m) = (0, e^{-\frac{it}{\xi}} \cdot m)  = -\frac{t}{\xi} \cdot (t, m).
	\]
	We now use the two above equalities to compute
	\begin{equation*}
		\begin{split}
			((i \circ \tilde{\varphi})^* \alpha)_{(t, m)} (u, v) &= \alpha_{-\frac{t}{\xi} \cdot (t, m)} (i_* \tilde{\varphi}_* (u, v)) \\ &= \alpha_{-\frac{t}{\xi} \cdot (t, m)} \left( 0, \left( e^{-\frac{it}{\xi}} \cdot \right)_* \left( X_{-\frac{u}{\xi}}(m) + v \right) \right) \\ &= \alpha_{-\frac{t}{\xi} \cdot (t, m)} \left( -\frac{t}{\xi} \cdot \right)_* (0, X_{-\frac{u}{\xi}}(m) + v ) \\ &= \alpha_{(t, m)} (0, X_{-\frac{u}{\xi}}(m) + v) \\ &= \alpha_{(t, m)} (u, v),
		\end{split}
	\end{equation*}
	where for the second to last equality we used the $\R$-invariance of $\alpha$ and for the last we used that it is basic, so $\alpha_{(t, m)}(u, X_{\frac{u}{\xi}}(m)) = 0$. This concludes the proof.
\end{proof}

\end{lemma}

\begin{proposition}
\label{prop:alpha_xiNuE}
	Let $M$ be a manifold with a free $S^1$ action. This determines an action of $S^1$ on $Q = S^1 \times M$,
	\[
	e^{it} \cdot (z, m) = (z, e^{it} \cdot m), \ \forall (z, m) \in Q.
	\]
	Let $\theta$ be the canonical $1$-form on $S^1$; we also denote by $\theta$ its pullback to $Q$.
	
	Then the reduction of $(T^*(S^1 \times M), \tilde{\omega} = \omega_{\tilde{\theta}}, \tilde{\theta})$ with respect to a non-zero value $\xi \in \mathbb{R}$ is
	\[
	(T^*(S^1 \times M))_\xi \simeq (S^1 \times \mu'^{-1}(\xi), \tilde{\omega}_{| S^1 \times \mu'^{-1}(\xi)}, \tilde{\theta}_{| S^1 \times \mu'^{-1}(\xi)}),
	\]
	where $\mu':T^*M \to \mathbb{R}$ is the symplectic momentum mapping determined by the action of $S^1$ on $T^*M$.
	
\begin{proof}
	Recall from \ref{prop:momentumCTG} that we have
	\begin{equation*}
	\begin{split}
	\mu: T^*(S^1 \times M) &\simeq S^1 \times \mathbb{R} \times T^*M \to \mathbb{R}, \\
	\mu(z, t, \alpha_m) &=  \mu'(\alpha_m) = -\alpha_m(X),
	\end{split} 
	\end{equation*}
	where we denote by $X = X_1$ the fundamental vector field corresponding to $1 \in \mathbb{R} \simeq \mathfrak{g}$. Hence, 	
	\[
	\mu^{-1}(\xi) = S^1 \times \mathbb{R} \times \mu'^{-1}(\xi).
	\]
	Recall that
	\[
	(T^*(S^1 \times M))_\xi \simeq \mu^{-1}(\xi) / \langle \tilde{X} + \xi \tilde{\theta}^{\tilde{\omega}} \rangle,
	\] 
	where $\tilde{X}$ is the fundamental vector field of the action of $S^1$ on $T^*M$.
	
	A quick computation shows that $\tilde{\theta}^{\tilde{\omega}}$ is the generator of $T \mathbb{R}$, so the action of $\R$ on $\mu^{-1}(\xi) = \R \times (S^1 \times \mu'^{-1}(\xi))$ given by the foliation associated to the distribution $\langle \tilde{X} + \xi \tilde{\theta}^{\tilde{\omega}} \rangle$ is precisely of the type described in \ref{lem:RxM}. Now, $\alpha = \tilde{\omega} \in \Omega^2(\mu^{-1}(\xi))$ is both $\R$-invariant and basic automatically, as the theory of LCS reduction (see \ref{th:LCSred}) already assures us it descends to the quotient.

	In conclusion, \ref{lem:RxM} gives us
	\[
	(T^*(S^1 \times M))_\xi \simeq (S^1 \times \mu'^{-1}(\xi), \omega_{\tilde{\theta} | S^1 \times \mu'^{-1}(\xi)}, \tilde{\theta}_{| S^1 \times \mu'^{-1}(\xi)}),
	\]
	as claimed.	
\end{proof}
	
\end{proposition}

Note that, in particular, this implies that the reduces LCS manifold is not necessarily a cotangent bundle.

\begin{example}
	Take $M = S^3$ with the standard action of $S^1$. According to \ref{thm:embedding},
	\[
	(T^*(S^1 \times S^3))_0 \simeq T^* (S^1 \times \mathbb{CP}^1).
	\] 
	We have the momentum map
	\[
	\mu': T^*S^3 \simeq S^3 \times \mathbb{R}^3 \to \mathbb{R}, \ \mu'(q, (a, b, c)) = c,
	\]
	where we have a canonical trivialization for the cotangent bundle $T^*S^3$ arising from the quaternionic structure of $S^3$.
	
	Then, according to \ref{prop:alpha_xiNuE}, 
	\[
	(T^*(S^1 \times S^3))_\xi \simeq S^1 \times S^3 \times \mathbb{R}^2 \not \simeq T^* (S^1 \times \mathbb{CP}^1).
	\]

\end{example}

\bigskip

As opposed to the case described in \ref{prop:alpha_xiNuE}, where \ref{thm:embedding} doesn't work, there is a situation in which an $\alpha_\xi$ as required can be constructed using the results from the symplectic case:

\begin{proposition}
	If there exists a function $f \in \mathcal{C}^\infty(Q)$ such that, for all $a \in \mathfrak{g}_\xi$, $X_a(f) = \xi(a)$, then an $\alpha_\xi$ satisfying the hypotheses of \ref{thm:embedding}, namely
	\begin{equation*}
	\begin{split}
	(\alpha_\xi)_q &\in \mu^{-1}(\xi), \ \forall q \in Q, \\
	\mathcal{L}_{X_a} \alpha_\xi &= \xi(a) \theta, \ \forall a \in G_\xi.
	\end{split}
	\end{equation*}
	exists.
	
	\begin{proof}
		Take an $\beta_\xi$ as in \ref{prop:alphaxiSymplectic} \ie satisfying
		\begin{equation*}
		\begin{split}
		(\beta_\xi)_q &\in \mu^{-1}(\xi), \ \forall q \in Q, \\
		\mathcal{L}_{X_a} \beta_\xi &= 0, \ \forall a \in G_\xi.
		\end{split}
		\end{equation*}

	Then, taking $\alpha_\xi = \beta_\xi + f\theta$, we have
	\begin{equation*}
	\begin{split}
	(\alpha_\xi)_q &\in \mu^{-1}(\xi), \ \forall q \in Q, \\
	\mathcal{L}_{X_a} \alpha_\xi &= \mathcal{L}_{X_a} (\beta_\xi + f \theta) = X_a(f) \theta = \xi(a) \theta, \ \forall a \in G_\xi.
	\end{split}
	\end{equation*}
	\end{proof}
\end{proposition}

\begin{remark}
	Note that for the hypotheses of the above theorem to apply in a useful way, neither $M$ nor $G$ can be compact. Otherwise, $df$ would be zero in at least one point along an orbit of $G$, and the equality $X_a(f) = \xi(a)$ could not hold for $\xi \neq 0$.
	
	One very simple case where the hypotheses are satisfied is for an action of $\R$ on a non-compact manifold $Q$, by picking $f$ to be a function that increases linearly along the (only) fundamental vector field.
\end{remark}

\bigskip

\textbf{Acknowledgments.} I would like to thank my advisor Liviu Ornea for his guidance and suggestions throughout.


\begin{thebibliography}{100}
	
	\bibitem[AM]{am} R. Abraham, J. E. Marsden, {\em Foundations of Mechanics, Second Edition}, Addison-Wesley Publishing Company, Inc. (1978).
	
	\bibitem[Baz]{baz} G. Bazzoni, {\em }EMS Surv. Math. Sci. 5 (1) (2018), 129-154.
	
	\bibitem[BGP]{bgp} F. Belgun, O. Goertsches, D. Petrecca, {\em Locally conformally symplectic convexity}, Journal of Geometry and Physics {\bf 135} (2019), 235-252.
	
	\bibitem[Br]{br} R. Bryant, {\em An introduction to Lie groups and symplectic geometry}, Geometry and quantum field theory, IAS/Park City Math. Series 1, American Mathematical society (1995), 5-181.
	
	\bibitem[EM]{el} Y. Eliashberg, E. Murphy, {\em Making cobordisms symplectic}, arXiv:1504.06312.
	
	\bibitem[HR1]{hr} S. Haller, T. Rybicki, {\em Reduction for locally conformal symplectic manifolds}, Journal of Geometry and Physics {\bf 37} (2001), 262-271.
	
	\bibitem[HR2]{hr2} S. Haller, T. Rybicki, {\em On the group of diffeomorphisms preserving a locally conformal symplectic structure}, Annals of Global Analysis and Geometry {\bf 17}, Issue 5 (1999), 475-502.
	
	\bibitem[I]{nico} N. Istrati, {\em A characterisation of toric LCK manifolds}, Journal of Symplectic Geometry {\bf 17} (2019), no. 5, 1297-1316.
	
	\bibitem[Lee]{lee} H.C. Lee, {\em A kind of even-dimensional differential geometry and its application to exterior calculus}, Amer. J. of Math. {\bf{65}} (1943), 433-438. 
	
	\bibitem[MW1]{mw1} J. E. Marsden, A. Weinstein, {\em Coadjoint orbits, vortices and Clebsch variables for incompressible fluids}, Physica D {\bf 7} (1983), 305-323.
	
	\bibitem[MW1]{mw1} J. E. Marsden, A. Weinstein, {\em The Hamiltonian structure of the Maxwell-Vlasov equations}, Physica D {\bf 4} (1982), 394-406.
	
	\bibitem[MW2]{mw2} J. E. Marsden, A. Weinstein, {\em Coadjoint orbits, vortices and Clebsch variables for incompressible fluids}, Physica D {\bf 7} (1983), 305-323.
	
	\bibitem[MS]{McD-S}D. McDuff, D. Salamon, \textit{Introduction to symplectic topology}, Oxford Mathematical Monographs. Oxford Science Publications. The Clarendon Press, Oxford University Press, New York, 1995.
	
	\bibitem[MMOPR]{ratiu} J. Marsden, G. Misiołek, J-P. Ortega, M. Perlmutter, T. Rațiu, {\em Hamiltonian Reduction by Stages}, Springer Lecture Notes in Mathematics, Vol. 1913 (2007).
	
	\bibitem[Sm]{smale} S. Smale, {\em Topology and Mechanics}, Inv. Math. {\bf{10}} (1970), 305-331; {\bf{11}} (1970), 45-64.
	
	\bibitem[St]{stnc} M. Stanciu, {\em Locally conformally symplectic reduction}, Annals of Global Analysis and Geometry {\bf 56} (2019), 245-275.
	
	\bibitem[V]{va_lcs} I. Vaisman, {\em Locally conformally symplectic manifolds}, Internat. J. Math. \& Math. Sci. {\bf 8} (1985), 521--536.
	
\end{thebibliography}
\end{document}